\documentclass[11pt, letterpaper,reqno]{amsart}

\usepackage{graphicx}
\usepackage[hyphens]{url}
\usepackage{appendix}

\numberwithin{equation}{section}

\DeclareMathOperator*{\argmax}{arg\,max}

\newtheorem{theorem}{Theorem}[section]
\newtheorem{assumption}{Assumption}[section]

\newtheorem{lemma}{Lemma}[section]
\newtheorem{proposition}{Proposition}[section]
\theoremstyle{definition}
\newtheorem{definition}{Definition}[section]
\theoremstyle{remark}
\newtheorem{remark}{Remark}[section]
\newtheorem{example}{Example}[section]

\newcommand{\leref}{Lemma~\ref}

\newcommand{\exref}{Example~\ref}
\newcommand{\prref}{Proposition~\ref}
\newcommand{\thref}{Theorem~\ref}

\renewcommand{\P}{\mathbb{P}}
\newcommand{\Q}{Q}
\newcommand{\R}{\mathbb{R}}

\newcommand{\N}{\mathbb{N}}

\newcommand{\cM}{\mathcal{M}}

\newcommand{\kP}{\mathfrak{P}}

\newcommand{\cI}{\mathcal{I}}
\newcommand{\cJ}{\mathcal{J}}
\newcommand{\eps}{\varepsilon}

\title[]{On hedging American options under model uncertainty}
\author[]{Erhan Bayraktar} \thanks{Erhan Bayraktar and Zhou Zhou are supported in part by the National Science Foundation under grant DMS-0955463. Yu-Jui Huang is supported in part by SFI (07/MI/008 and 08/SRC/FMC1389) and the ERC (278295).}  
\address{Department of Mathematics, University of Michigan}
\email{erhan@umich.edu}
\author[]{Yu-Jui Huang}
\address{School of Mathematical Sciences, Dublin City University}
\email{yujui.huang@dcu.ie}
\author[]{Zhou Zhou}
\address{Department of Mathematics, University of Michigan}
\email{zhouzhou@umich.edu}
\date{First version: September 11, 2013. This version: January 31, 2015.}
\begin{document}
\maketitle

\begin{abstract}
We consider as given a discrete time financial market with a risky asset and options written on that asset and determine both the sub- and super-hedging prices of an American option in the model independent framework of \cite{Nutz2}. We obtain the duality of results for the sub- and super-hedging prices, as well as the existence of the optimal hedging strategies. For the sub-hedging prices we discuss whether the sup and inf in the dual representation can be exchanged (a counter example shows that this is not true in general). For the super-hedging prices we discuss several alternative definitions and argue why our choice is more reasonable.
 Then assuming that the path space is compact, we construct a discretization of the path space and demonstrate the convergence of the hedging prices at the optimal rate. The latter result would be useful for numerical computation of the hedging prices. Our results generalize those of \cite{DS2} to the case when static positions in (finitely many) European options can be used in the hedging portfolio.
\end{abstract}
\section{Introduction}
We consider the problem of pricing and semi-static hedging of American options in the model uncertainty set-up of \cite{Nutz2}. In semi-static hedging stocks are traded dynamically and options are traded statically. This formulation is frequently used in the literature since options are less liquid than stocks (see e.g. \cite{MR2281789}). In this setting, so far only the super-hedging prices of (path dependent) European options under (non-dominated) model uncertainty were considered: see e.g. \cite{Schachermayer2}, \cite{Mathias} and \cite{Nutz2}. \cite{DS13} obtained these results for a continuous time financial market.  Some results are available on the pricing of American options in the model independent framework without the static hedging in options. See for example \cite{DS2} for duality results in discrete time set-up, and \cite{2013arXiv1301.0091B, 2012arXiv1209.6601E, 2012arXiv1212.2140N} for similar duality results and in particular the analysis of the related optimal stopping problem.

In this paper, we consider the problems of sub- and super-hedging of American options using semi-static trading strategies in the model independent set-up of \cite{Nutz2}. We first obtain the duality results for both the sub- and super-hedging prices, as well as the existence of the optimal hedging strategies. Then for compact state spaces we show how to discretize it in order to obtain the optimal rate of convergence.

In the first part of this paper, we focus on the sub-and super-hedging dualities. For the sub-hedging prices we discuss whether the sup and inf in the dual representation can be exchanged. We show that the exchangeability may fail in general unless there is no hedging option. For the super-hedging prices we discuss several alternative definitions. The correct definition involves ``non-anticipative" strategies, which is quite different from the one in the classical case when there is no hedging option. As for the existence for the optimal hedging strategies, we first develop a new proof to obtain the existence of an optimal static hedge. Then we use the non-dominated optimal stopping to obtain the optimal trading strategy in the stock for sub-hedging problem, and the optional decomposition for super-hedging. 

In the second part of this paper, we concentrate on how to use hedging prices in the discretized market to approximate the ones in the original market. This approximation is useful for numerical computations since in the discretized market the state space is finite, and thus there exists a dominating measure on it. Our approximation result is a generalization of \cite{DS2}, but in our case the construction of the approximation becomes much more complicated due to the presence of the hedging options. In particular, in contrast to \cite{DS2}, it is not a priori clear that the discretized market is free of arbitrage. We also show how to pick the prices of the hedging options in the discretized market in order to obtain the optimal convergence rate.
One should note that, although in \cite{DS2}  the no-arbitrage notions of  \cite{Schachermayer2} and  \cite{Nutz2} coincide (see Appendix \ref{section:appD}), in our case they are different since there are hedging options available. We choose to work in the framework of \cite{Nutz2}. 

The rest of the paper is organized as follows: We obtain the duality results for the sub- and super-hedging prices of American options in Sections 2 and 3, respectively.  In Section 4, we discretize the path space and show that hedging prices in the discretized market converge to the original ones. The appendix is devoted to verify some of the statements we make in Sections 1, 2 and 3. Of particular interest, in that section, is the analysis of the adverse optimal stopping problems for nonlinear expectations in discrete time, which resolves the optimal stopping problems in \cite{DS2} for more general state spaces (see Appendix~\ref{sec:appB}). This result is useful particularly in showing the existence of the optimal sub-hedging strategy. The existence of the optimal super hedging strategy is a consequence of the non-dominated optional decomposition theorem \cite{Nutz2} and the analysis in Appendix \ref{sec:appC}.

The remainder of this section is devoted to setting up the notation used in the rest of the paper. 

\subsection{Notation}\label{sec:set-up}\label{notation}

We use the set-up in \cite{Nutz2}. Let $T\in\mathbb{N}$ be the time Horizon and let $\Omega_1$ be a Polish space. For $t\in\{0,1,\dotso,T\}$, let $\Omega_t:=\Omega_1^t$ be the $t$-fold Cartesian product, with the convention that $\Omega_0$ is a singleton. We denote by $\mathcal{F}_t$ the universal completion of $\mathcal{B}(\Omega_t)$ and write $(\Omega,\mathcal{F})$ for $(\Omega_T,\mathcal{F}_T)$. For each $t\in\{0,\dotso,T-1\}$ and $\omega\in\Omega_t$, we are given a nonempty convex set $\mathcal{P}_t(\omega)\subset\mathfrak{P}(\Omega_1)$ of probability measures.  Here $\mathcal{P}_t$ represents the possible models for the $t$-th period, given state $\omega$ at time $t$. We assume that for each $t$, the graph of $\mathcal{P}_t$ is analytic, which ensures that $\mathcal{P}_t$ admits a universally measurable selector, i.e., a universally measurable kernel $P_t:\ \Omega_t\rightarrow \mathfrak{P}(\Omega_t)$ such that $P_t(\omega)\in\mathcal{P}_t(\omega)$ for all $\omega\in\Omega_t$. Let
\begin{equation}\label{prob}
\mathcal{P}:=\{P_0\otimes\dotso\otimes P_{T-1}:\ P_t(\cdot)\in\mathcal{P}_t(\cdot),\ t=0,\dotso,T-1\},
\end{equation}
where each $P_t$ is a universally measurable selector of $\mathcal{P}_t$, and
$$P_0\otimes\dotso\otimes P_{T-1}(A)=\int_{\Omega_1}\dotso\int_{\Omega_1} 1_A(\omega_1,\dotso,\omega_T)P_{T-1}(\omega_1,\dotso,\omega_{T-1};d\omega_T)\dotso P_0(d\omega_1),\ \ \ A\in\Omega.$$

Let $S_t:\Omega_t\rightarrow\mathbb{R}$ be Borel measure, which represents the price at time $t$ of a stock $S$ that can be traded dynamically in the market.  Let $g=(g_1,\dotso,g_e):\Omega\rightarrow\mathbb{R}^e$ be Borel measurable, representing the options that can only be traded at the beginning at price $0$. Assume NA$(\mathcal{P})$ holds, i.e, for all $(H,h)\in\mathcal{H}\times\mathbb{R}^e$,
$$(H\cdot S)_T+hg\geq 0\ \ \ \mathcal{P}-\text{q.s.\ \ \ \ \ \ implies \ \ \ \ \ \ } (H\cdot S)_T+hg=0\ \ \ \mathcal{P}-\text{q.s.},$$
where $\mathcal{H}$ is the set of predictable processes representing trading strategies, $(H\cdot S)_T=\sum_{t=0}^{T-1}H_t(S_{t+1}-S_t)$, and $h g$ denotes the inner product of $h$ and $g$. Then from \cite[FTAP]{Nutz2}, for all $P\in\mathcal{P}$, there exists $Q\in\mathcal{Q}$ such that $P\ll Q$, where 
$$\mathcal{Q}:=\{Q \text{ martingale measure}\footnote{That is, $Q$ satisfies $E_Q[|S_{t+1}|\ |\mathcal{F}_t]<\infty$ and $E_Q[S_{t+1}|\mathcal{F}_t]=S_t, \ Q$-a.s. for $t=0,\dotso,T-1$.}:\ E_Q[g^i]=0,\ i=1,\dotso,e, \text{ and }\exists P'\in\mathcal{P}, \text{ s.t. } Q\ll P'\}.$$
In the next section we will consider an American option with pay-off stream $\Phi$. We will assume that $\Phi:\{0,\dotso,T\}\times\Omega\rightarrow\mathbb{R}$ is adapted\footnote{Unless otherwise specified the measurability and related concepts (adaptedness, etc) are with respect to the filtration $(\mathcal{F}_t)_{t=0}^T$.}. Let $\mathcal{T}$ be the set of stopping times with respect to the {\em{raw}} filtration $(\mathcal{B}(\Omega_t))_{t=0}^T$, and $\mathcal{T}_t\subset\mathcal{T}$ the set of stopping times that are no less than $t$.
For $t=0,\dotso,T$ and $\omega\in\Omega_t$, define 
$$\mathcal{Q}_t(\omega):=\{Q\in\mathfrak{P}(\Omega_1):\ Q\ll P, \text{ for some } P\in\mathcal{P}_t(\omega),\text{ and } E_Q[\Delta S_{t+1}(\omega,\cdot)]=0\}.$$
By \cite[Lemma 4.8]{Nutz2}, there exists a universally measurable selector $Q_t$ such that $Q_t(\cdot)\in\mathcal{Q}_t(\cdot)$ on $\{\mathcal{Q}_t\neq\emptyset\}$. Using these selectors we define for $t\in\{0,\dotso,T-1\}$ and $\omega\in\Omega_t$,
$$\mathcal{M}_t(\omega):=\left\{Q_t\otimes\dotso\otimes Q_{T-1}:\ Q_i(\omega,\cdot)\in\mathcal{Q}_i(\omega,\cdot)\text{ on }\{\mathcal{Q}_i(\omega,\cdot)\neq\emptyset\},\ i=t,\dotso,T-1\right\},$$
which is similar to \eqref{prob} but starting from time $t$ instead of time 0. In particular $\mathcal{M}_0=\mathcal{M}$, where 
\begin{equation}\label{123}
\mathcal{M}:=\{Q \text{ martingale measure}: \exists P\in\mathcal{P}, \text{ s.t. }Q\ll P\}.
\end{equation} 
We will assume in the rest of the paper that the graph of $\mathcal{M}_t$ is analytic, $t=0,\dotso,T-1$. Below we provide a general sufficient condition for the analyticity of graph($\mathcal{M}_t)$ and leave its proof to Appendix~\ref{sec:appA}. 
\begin{proposition}\label{prop 1.1}
For $t=0,\dotso,T-1$ and $\omega\in\Omega_t$, define 
$$\mathbb{P}_t(\omega):=\{P_t\otimes\dotso\otimes P_{T-1}:\ P_i(\omega,\cdot)\in\mathcal{P}_i(\omega,\cdot),\ i=t,\dotso,T-1\},$$
where each $P_i$ is a universally measurable selector of $\mathcal{P}_i$. If graph($\mathbb{P}_t$) is analytic, then graph($\mathcal{M}_t$) is also analytic. 
\end{proposition}

For any measurable function $f$ and probability measure $P$, we define the $P$-expectation of $f$ as $E_P[f]=E_P[f^+]-E_P[f^-]$ with convention $\infty-\infty=-\infty$. We use $|\cdot|$ to denote the sup norm in various cases. For $\omega\in\Omega$ and $t\in\{0,\dotso,T\}$, we will use the notation $\omega^t\in\Omega_t$ to denote the path up to time $t$. For a given function $f$ defined on $\Omega$, let us denote
$$\underline{\mathcal{E}}_\tau(f)(\omega):=\inf_{Q\in\mathcal{M}_{\tau(\omega)}(\omega^{\tau(\omega)})}E_Q[f(\omega^{\tau(\omega)},\cdot)],\quad \omega\in\Omega,$$
and
$$\overline{\mathcal{E}}_\tau(f)(\omega):=\sup_{Q\in\mathcal{M}_{\tau(\omega)}(\omega^{\tau(\omega)})}E_Q[f(\omega^{\tau(\omega)},\cdot)],\quad \omega\in\Omega.$$

We use the abbreviations u.s.a. for upper-semianalytic, l.s.c. for lower-semicontinuous, and u.s.c. for upper-semicontinuous. 

\section{Sub-hedging Duality}
\noindent We define the sub-hedging price of the American option as
\begin{equation}\label{sub}
\underline\pi(\Phi):=\sup\left\{x\in\mathbb{R}:\ \exists (H,\tau,h)\in\mathcal{H}\times\mathcal{T}\times\mathbb{R}^e,\ \text{ s.t. } \Phi_\tau+(H\cdot S)_T+hg\geq x,\ \mathcal{P}-q.s.\right\}.
\end{equation}
\begin{remark}
In the above definition, we require the trading in the stock $S$ to be up to time $T$ instead of $\tau$. This is because  it is possible that the maturities of some options in $g$ are later than $\tau$. When there is no hedging options involved, for sub-hedging (and in fact also super-hedging) trading $S$ up to time $T$ is equivalent to up to time $\tau$ (e.g. see the beginning of the proof of \cite[Proposition 6.1]{ZZ}).
\end{remark}
We have the following duality theorem for sub-hedging prices.
\begin{theorem}\label{sub}
Assume that $\Phi_t$ is l.s.a. for $t=1,\dotso,T$. Then
\begin{equation}\label{dual1}
\underline\pi(\Phi)=\sup_{\tau\in\mathcal{T}}\inf_{Q\in\mathcal{Q}}E_Q[\Phi_\tau].
\end{equation}
Moreover, if $\sup_{Q\in\mathcal{M}}E_Q[|g|]<\infty$, $\sup_{Q\in\mathcal{M}}E_Q[\max_{0\leq t\leq T}|\Phi_t|]<\infty$, and for any $h\in\mathbb{R}^e$ and $t\in\{0,\dotso,T-1\}$, the maps $\Phi_t+\underline{\mathcal{E}}_t(hg)$ and  $\phi:\ \Omega\mapsto\mathbb{R}^e$ defined by
$$\phi=\underline{\mathcal{E}}_t\left(\inf_{\tau\in\mathcal{T}_{t+1}}\underline{\mathcal{E}}_{t+1}\left(\Phi_\tau+\underline{\mathcal{E}}_\tau(hg)\right)\right)\quad \left(\text{or }\phi=\underline{\mathcal{E}}_t\left(\sup_{Q\in\mathcal{M}_{t+1}}\inf_{\tau\in\mathcal{T}_{t+1}}E_Q\left(\Phi_\tau+\underline{\mathcal{E}}_\tau(hg)\right)\right)\right)$$
are Borel measurable, then there exists $(H^*,\tau^*,h^*)\in\mathcal{H}\times\mathcal{T}\times\mathbb{R}^e$, such that
\begin{equation}\label{aadd}
\Phi_{\tau^*}+(H^*\cdot S)_T+h^*g\geq\pi(\Phi),\quad\mathcal{P}-q.s.
\end{equation}
\end{theorem}
\begin{proof}
For any $\tau\in\mathcal{T}$, define
$$\underline\pi(\Phi_\tau):=\sup\left\{x\in\mathbb{R}:\ \exists (H,h)\in\mathcal{H}\times\mathbb{R}^e,\ \text{ s.t. } \Phi_\tau+(H\cdot S)_T+hg\geq x,\ \mathcal{P}-q.s.\right\}.$$
Since $\Phi_t$ is u.s.a. and $\tau$ is a stopping time with respect to the raw filtration, it follows that $\Phi_\tau$ is u.s.a. Then applying \cite[Theorem 5.1 (b)]{Nutz2}, we get
$$\underline\pi(\Phi_\tau)=\inf_{Q\in\mathcal{Q}}\mathbb{E}_Q[\Phi_\tau]\ \ \ \Longrightarrow\ \ \ \sup_{\tau\in\mathcal{T}}\underline\pi(\Phi_\tau)=\sup_{\tau\in\mathcal{T}}\inf_{Q\in\mathcal{Q}}\mathbb{E}_Q[\Phi_\tau].$$
Since $\underline\pi(\Phi)\geq\underline\pi(\Phi_\tau),\ \forall\tau\in\mathcal{T}$, it follows that $\underline\pi(\Phi)\geq\sup_{\tau\in\mathcal{T}}\underline\pi(\Phi_\tau)$. Therefore, it remains to show that $\underline\pi(\Phi)\leq\sup_{\tau\in\mathcal{T}}\underline\pi(\Phi_\tau)$. For any $\eps>0$, there exists $x\in(\underline\pi(\Phi)\wedge(1/\eps)-\eps,\underline\pi(\Phi)\wedge(1/\eps)]$ and $(H^\eps,\tau^\eps,h^\eps)\in\mathcal{H}\times\mathcal{T}\times\mathbb{R}^e$ satisfying
$$\Phi_{\tau^\eps}+(H^\eps\cdot S)_T+h^\eps g\geq x,\ \mathcal{P}-q.s.$$
As a result,
$$\underline\pi(\Phi)\wedge\frac{1}{\eps}-\eps< x\leq \underline\pi(\Phi_{\tau^\eps})\leq\sup_{\tau\in\mathcal{T}}\underline\pi(\Phi_\tau),$$
from which \eqref{dual1} follows since $\eps$ is arbitrary.

Let us turn to the proof of the existence of the optimal sub-hedging strategies. Similar to the proof above, we can show that
\begin{equation}\notag
\underline\pi(\Phi)=\sup_{h\in\mathbb{R}^e}\sup_{\tau\in\mathcal{T}}\sup\{x:\ \exists H\in\mathcal{H},\ \text{s.t.}\ \Phi_\tau+(H\cdot S)_T+hg\geq x,\ \mathcal{P}-q.s.\}=\sup_{h\in\mathbb{R}^e}\sup_{\tau\in\mathcal{T}}\inf_{Q\in\mathcal{M}}E_Q[\Phi_\tau+hg].
\end{equation}
We shall first show in two steps that the optimal $h^*$ exists for the above equations.\\
\textbf{Step 1}: We claim that $0$ is in the relative interior of the convex set $\{E_Q[g],\ \Q\in\mathcal{M}\}$. If not, then there exists $h\in\mathbb{R}^e$, such that $E_Q[hg]\leq 0,$ for any $Q\in\mathcal{M}$, and moreover there exists $\bar Q\in\mathcal{M}$, such that $E_{\bar Q}[hg]<0$. By \cite[Theorem 4.9]{Nutz2}, the super-hedging price of $hg$ (using only the stock) is $\sup_{Q\in\mathcal{M}}E_Q[hg]\leq 0$, and there exists $H\in\mathcal{H}$, such that
$$(H\cdot S)_T\geq hg,\quad\mathcal{P}-q.s.$$
Then $E_{\bar Q}[(H\cdot S)_T-hg]>0$, and thus, for any $P\in\mathcal{P}$ dominating $\bar Q$, we have that
$$P((H\cdot S)_T-hg>0)>0,$$
which contradicts NA$(\mathcal{P})$. \\
\textbf{Step 2}: Since  $0$ is a relative interior point of $\{E_Q[g],\ \Q\in\mathcal{M}\}$, and $\sup_{Q\in\mathcal{M}}E_Q[\max_{0\leq t\leq T}|\Phi_t|]<\infty$, we know that 
$$\underline\pi(\Phi)=\sup_{h\in\mathbb{R}^e}\sup_{\tau\in\mathcal{T}}\inf_{Q\in\mathcal{M}}E_Q[\Phi_\tau+hg]=\sup_{h\in\mathbb{K}}\sup_{\tau\in\mathcal{T}}\inf_{Q\in\mathcal{M}}E_Q[\Phi_\tau+hg]$$
where $\mathbb{K}$ is a compact subset  of $\mathbb{R}^e$. Define the map $\varphi:\ \mathbb{R}^e\mapsto\mathbb{R}$ by
$$ \varphi(h)=\sup_{\tau\in\mathcal{T}}\inf_{Q\in\mathcal{M}}E_Q[\Phi_\tau+hg].$$
The function $\varphi$ is continuous since $|\varphi(h)-\varphi(h')|\leq e|h-h'|\sup_{Q\in\mathcal{M}}E_Q|g|$. Hence, there exists $h^*\in\mathbb{K}\subset\mathbb{R}^e$ such that
\begin{equation}\label{aacc}
\underline\pi(\Phi)=\sup_{\tau\in\mathcal{T}}\inf_{Q\in\mathcal{M}}E_Q[\Phi_\tau+h^*g]=\sup_{\tau\in\mathcal{T}}\inf_{Q\in\mathcal{M}}E_Q[\Phi_\tau+\underline{\mathcal{E}}_\tau(h^*g)],
\end{equation}
where the second equality above follows from  \cite[Theorem 2.3]{Nutz3}. Using the measurability assumptions in the statement of this theorem, we can apply \thref{infsup}, and obtain a $\tau^*\in\mathcal{T}$ that is optimal for \eqref{aacc}, i.e.,
\begin{eqnarray}
\underline\pi(\Phi)&=&\inf_{Q\in\mathcal{M}}E_Q[\Phi_{\tau^*}+\underline{\mathcal{E}}_{\tau^*}(h^*g)]=\sup_{\tau\in\mathcal{T}}\inf_{Q\in\mathcal{M}}E_Q[\Phi_\tau+h^*g]\notag\\
&=&\sup\{x:\ \exists H\in\mathcal{H},\ \text{s.t.}\ \Phi_{\tau^*}+(H\cdot S)_T+h^*g\geq x,\ \mathcal{P}-q.s.\}
\end{eqnarray}
Then by \cite[Theorem 4.9]{Nutz2}, there exists a strategy $H^*\in\mathcal{H}$, such that \eqref{aadd} holds.
\end{proof}

\subsection{Exchangeability of the supremum and infimum in \eqref{dual1}}

When there are no options available for static hedging (then $\mathcal{Q}=\mathcal{M}$), $\mathcal{Q}$ is closed under pasting. Using this property we show in \thref{infsup}  and \prref{exa} that the order of \lq\lq inf\rq\rq\ and \lq\lq sup\rq\rq\ in \eqref{dual1} can be exchanged under some reasonable assumptions. These conclusions cover the specific results of \cite{DS2} which works with a compact path space. (Although, our no arbitrage assumption seems to be different than the one in \cite{DS2}, we verify in Proposition~\ref{prop:no-opt} that they are the same when there are no options, i.e., $e=0$.) The same holds true for our super-hedging result in the next section. 

In general, $\mathcal{Q}$ may not be stable under pasting due to the distribution constraints imposed by having to price the given  options correctly. Then whether  the \lq\lq inf\rq\rq\ and \lq\lq sup\rq\rq\ in \eqref{dual1} can be exchanged is not clear, and in fact may not be possible as the example below demonstrates.
\begin{center}
\includegraphics[scale=0.5]{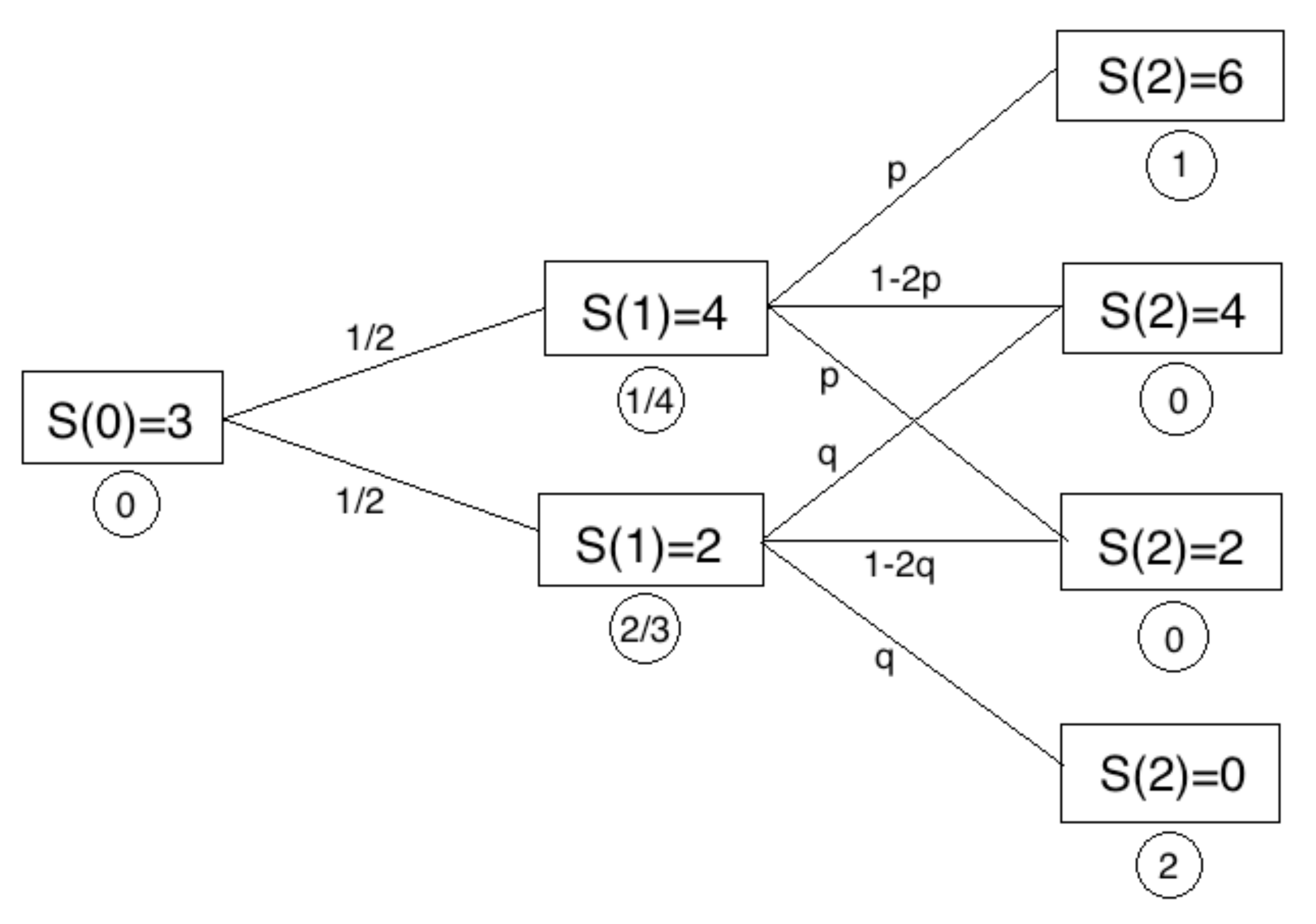}
\end{center}
\begin{example}\label{counterexample}
We consider a two-period model as described by the figure above. The stock price process is restricted to the finite path space indicated by the graph, where $S(t)$ is the stock price at time $t,\ t=0,1,2$. Let $\mathcal{P}$ be all the probability measures on this path space. Then each martingale measure $Q\in\mathcal{M}$ can be uniquely characterized by a pair $(p,q),\ 0\leq p,q\leq 1/2$, as indicated in the graph. Assume there is one European option $g=[S(2)-3]^+-5/6$ that can be traded at price $0$. Let $\Phi$ be the payoff of a path-independent American option that needs to be hedged. In the graph, the number in each circle right below the rectangle (node) represents the value of $\Phi$ when the stock price is at that node.

Each $Q\in\mathcal{Q}\subset\mathcal{M}$ is characterized by $(p,q)$ with the additional condition: $p+q=2/3$. There are in total 5 stopping strategies: stop at node $S(0)=3$, or continue to node $S(1)=k,\ k=2,4$, then choose either to stop or to continue. It is easy to check that
\begin{equation}\notag
\sup_{\tau\in\mathcal{T}}\inf_{Q\in\mathcal{Q}}E_Q[\Phi_\tau]=0\vee\frac{11}{24}\vee\left(\frac{1}{8}+\inf_{\substack{0\leq p, q\leq 1/2\\p+q=2/3}}q\right)\vee\left(\inf_{\substack{0\leq p, q\leq 1/2\\p+q=2/3}}\frac{p}{2}+\frac{1}{3}\right)\vee\left(\inf_{\substack{0\leq p, q\leq 1/2\\p+q=2/3}}\left(\frac{p}{2}+q\right)\right)=\frac{11}{24},
\end{equation}
and
\begin{equation}\notag
\inf_{Q\in\mathcal{Q}}\sup_{\tau\in\mathcal{T}}E_Q[\Phi_\tau]=\inf_{\substack{0\leq p, q\leq 1/2\\p+q=2/3}}\left[\frac{1}{2}\left(p\vee\frac{1}{4}+2q\vee\frac{2}{3}\right)\vee0\right]=\frac{1}{2}>\sup_{\tau\in\mathcal{T}}\inf_{Q\in\mathcal{Q}}E_Q[\Phi_\tau].
\end{equation}
\end{example}

\section{The Super-hedging Duality}
\noindent We define the super-hedging price as
\begin{equation}\label{sup}
\overline\pi(\Phi):=\inf\left\{x\in\mathbb{R}:\ \exists (H,h)\in\mathcal{H}'\times\mathbb{R}^e,\ \text{ s.t. } x+(H\cdot S)_T+hg\geq \Phi_\tau,\ \mathcal{P}-q.s.,\ \forall\tau\in\mathcal{T}\right\},
\end{equation}
where $\mathcal{H}'$ is the set of processes that have the \lq\lq  non-anticipativity\rq\rq\ property, i.e.,
\begin{equation}\label{nonanti}
\mathcal{H}':=\{H:\mathcal{T}\mapsto \mathcal{H},\text{ s.t. }H_t(\tau^1)=H_t(\tau^2),\ \forall t<\tau^1\wedge\tau^2\}.
\end{equation}
In other words, the seller of the American option is allowed to adjust the trading strategy according to the stopping time $\tau$ after it is realized. 

The following is our duality theorem for the super-hedging prices.
\begin{theorem}\label{supheg}
Assume that for $(\omega,P)\in\Omega_T\times\mathfrak{P}(\Omega_{T-t})$, 
\begin{equation}\label{cc10}
\text{the map } (\omega,P)\mapsto\sup_{\tau\in\mathcal{T}_t}E_P[\Phi_\tau(\omega^t,\cdot)] \text{ is u.s.a.},\  t=1,\dotso,T.
\end{equation}
Then
\begin{equation}\label{supdual}
\overline\pi(\Phi)=\inf_{h\in\mathbb{R}^e}\sup_{\tau\in\mathcal{T}}\sup_{Q\in\mathcal{M}}E_Q[\Phi_\tau-hg].
\end{equation}
Moreover, if $\sup_{Q\in\mathcal{M}}E_Q[|g|]<\infty$ and $\sup_{Q\in\mathcal{M}}E_Q[\max_{0\leq t\leq T}|\Phi_t|]<\infty$, then there exists $(H^*,h^*)\in\mathcal{H}'\times\mathbb{R}^e$, such that
\begin{equation}\label{aaee}
\overline\pi(\Phi)+(H^*\cdot S)_T+h^*g\geq\Phi_\tau,\ \mathcal{P}-q.s.,\ \forall\tau\in\mathcal{T}.
\end{equation}
\end{theorem}

\begin{proof}
An argument similar to the one used in the proof of \thref{sub} implies that $\overline\pi(\Phi)=\inf_{h\in\mathbb{R}^e}\overline\pi(\Phi,h)$, where
$$\overline\pi(\Phi,h)=\inf\left\{x\in\mathbb{R}:\ \exists H\in\mathcal{H}',\ \text{ s.t. } x+(H\cdot S)_T+hg\geq \Phi_\tau,\ \mathcal{P}-q.s.,\ \forall\tau\in\mathcal{T}\right\}.$$
It is easy to see that $\overline\pi(\Phi,h)\geq\sup_{\tau\in\mathcal{T}}\sup_{Q\in\mathcal{M}}E_Q[\Phi_\tau-hg]$. In what follows we will demonstrate the reverse inequality. Define
\begin{equation}\label{V}
V_t=\sup_{\tau\in\mathcal{T}_t}\overline{\mathcal{E}}_t(\Phi_\tau-hg).
\end{equation}
Using assumption \eqref{cc10}, we apply \prref{cc3} to show that $V_t$ is u.s.a., $\mathcal{F}_t$-measurable and a super-martingale under each $Q \in \mathcal{M}$. As a result, we can apply the optional decomposition theorem for the nonlinear expectations \cite[Theorem 6.1]{Nutz2}, which implies that there exists $H'\in\mathcal{H}$, such that for any $\tau\in\mathcal{T}$, 
\begin{equation}\label{op}
V_0+(H'\cdot S)_\tau\geq V_\tau=\sup_{\rho\in\mathcal{T}_\tau}\overline{\mathcal{E}}_\tau(\Phi_\rho-hg)\geq\Phi_\tau+\overline{\mathcal{E}}_\tau(-hg),\quad \mathcal{P}-q.s.
\end{equation} 
Let us also define
$$W_t:=\overline{\mathcal{E}}_t(-hg).$$
Thanks to \prref{cc3}, we can apply \cite[Theorem 6.1]{Nutz2} again and get that there exists $H''\in\mathcal{H}$, such that for any $\tau\in\mathcal{T}$,
\begin{equation}\label{2}
W_\tau+(H''\cdot S)_{\tau,T}=\overline{\mathcal{E}}_\tau(-hg)+(H''\cdot S)_{\tau,T}\geq W_T=-hg,\quad \mathcal{P}-q.s.,
\end{equation}
where $(H''\cdot S)_{\tau,T}=\sum_{i=\tau}^{T-1} H_i''[S_{i+1}-S_i]$. Combining \eqref{op} and \eqref{2}, we get that
$$V_0+(H\cdot S)_T+hg\geq\Phi_\tau,\ \forall \tau\in\mathcal{T},\ \mathcal{P}-q.s.,$$
where $H_t=H_t'1_{\{t<\tau\}}+H_t''1_{\{t\geq\tau\}}$. Note that $H'$ in \eqref{op} is independent of $\tau$, which implies that $H$ is indeed in $\mathcal{H}'$. Hence, $V_0=\sup_{\tau\in\mathcal{T}}\sup_{Q\in\mathcal{M}}E_Q[\Phi_\tau-hg]\geq\overline\pi(\Phi,h)$.

As in the proof of \thref{sub}, there exists $h^*\in\mathbb{R}^e$ that is optimal for \eqref{supdual}:
$$\overline\pi(\Phi)=\sup_{\tau\in\mathcal{T}}\sup_{Q\in\mathcal{M}}E_Q[\Phi_\tau-h^*g]=\overline\pi(\Phi,h^*).$$
Also observe from the proof above that there exists $H^*\in\mathcal{H}'$, such that
$$\overline\pi(\Phi,h^*)+(H^*\cdot S)_T+h^*g\geq \Phi_\tau,\quad\mathcal{P}-q.s.,\ \forall\tau\in\mathcal{T},$$
which implies \eqref{aaee}.
\end{proof}

\begin{proposition}[A sufficient condition on the assumption \eqref{cc10} of \thref{supheg}]\label{semicti}
Assume that $\Phi_t$ is l.s.c. and bounded from below for $t=1,\dotso,T$. Then for $(\omega,P)\in\Omega_T\times\mathfrak{P}(\Omega_{T-t})$, the map $(\omega,P)\mapsto\sup_{\tau\in\mathcal{T}_t}E_P[\Phi_\tau(\omega^t,\cdot)]$ is l.s.c., and thus u.s.a, $t=1,\dotso,T$.
\end{proposition}
\begin{proof}
If $\Phi$ is uniformly continuous in $\omega$ with modulus of continuity $\rho$, then for $(^n\omega,P^n)\rightarrow(\omega,P)$, we have that
\begin{eqnarray}
&&\sup_{\tau\in\mathcal{T}_t}E_{P^n}[\Phi_\tau((^n\omega)^t,\cdot)]-\sup_{\tau\in\mathcal{T}_t}E_P[\Phi_\tau(\omega^t,\cdot)]\notag\\
&&=\sup_{\tau\in\mathcal{T}_t}E_{P^n}[\Phi_\tau((^n\omega)^t,\cdot)]-\sup_{\tau\in\mathcal{T}_t}E_{P^n}[\Phi_\tau(\omega^t,\cdot)]+\sup_{\tau\in\mathcal{T}_t}E_{P^n}[\Phi_\tau(\omega^t,\cdot)]-\sup_{\tau\in\mathcal{T}_t}E_P[\Phi_\tau(\omega^t,\cdot)]\notag\\
&&\geq-\rho(||^n\omega-\omega||)+\sup_{\tau\in\mathcal{T}_t}E_{P^n}[\Phi_\tau(\omega^t,\cdot)]-\sup_{\tau\in\mathcal{T}_t}E_P[\Phi_\tau(\omega^t,\cdot)].\label{12}
\end{eqnarray}
Noting that the map $P\mapsto\sup_{\tau\in\mathcal{T}_t}E_P[\Phi_\tau(\omega^t,\cdot)]$ is l.s.c. (e.g, see in \cite[Theorem 1.1]{Elton89}), we know that the map $(P,\omega)\mapsto\sup_{\tau\in\mathcal{T}_t}E_P[\Phi_\tau(\omega^t,\cdot)]$ is l.s.c. by taking the limit in \eqref{12}. In general, if $\Phi_t$ be l.s.c. and bounded from below, then there exists uniformly continuous functions $(\Phi_t^n)_n$, such that $\Phi_t^n\nearrow \Phi_t$ pointwise (see e.g., \cite[Lemma 7.14]{Shreve}), $t=1,\dotso,T$. Therefore,
$$\sup_{\tau\in\mathcal{T}_t}E_P[\Phi_\tau(\omega^t,\cdot)]=\sup_{\tau\in\mathcal{T}_t}\sup_n E_P[\Phi_\tau^n(\omega^t,\cdot)]=\sup_n\sup_{\tau\in\mathcal{T}_t}E_P[\Phi_\tau^n(\omega^t,\cdot)],$$
which implies that the map $(\omega,P)\mapsto \sup_{\tau\in\mathcal{T}_t}E_P[\Phi_\tau(\omega^t,\cdot)]$ is l.s.c.
\end{proof}

\subsection{Comparison of several definitions of super-hedging}
In the duality result \eqref{supdual}, one would expect that $\overline\pi(\Phi)=\sup_{\tau\in\mathcal{T}}\sup_{Q\in\mathcal{Q}}E_Q[\Phi_\tau]$. More precisely, if the orders in \eqref{supdual} could be exchanged for then we would have 
\begin{equation}\notag
\overline\pi(\Phi)=\inf_{h\in\mathbb{R}^e}\sup_{\tau\in\mathcal{T}}\sup_{Q\in\mathcal{M}}E_Q[\Phi_\tau-hg]=\sup_{\tau\in\mathcal{T}}\sup_{Q\in\mathcal{M}}\inf_{h\in\mathbb{R}^e}E_Q[\Phi_\tau-hg]=\sup_{\tau\in\mathcal{T}}\sup_{Q\in\mathcal{Q}}E_Q[\Phi_\tau].
\end{equation}
But the latter is in fact equal to
\begin{equation}\label{eq:anshdf}
\hat\pi(\Phi):=\inf\{x\in\mathbb{R}:\ \forall\tau\in\mathcal{T},\exists(H,h)\in\mathcal{H}\times\mathbb{R}^e, \text{ s.t. } x+(H\cdot S)_T+hg\geq\Phi_\tau,\ \mathcal{P}-q.s.\}.
\end{equation}
That is,
\begin{equation}\label{e123}
\hat\pi(\Phi)=\sup_{\tau\in\mathcal{T}}\sup_{Q\in\mathcal{Q}}E_Q[\Phi_\tau].
\end{equation}
Since for the definition of $\hat\pi$ in \eqref{eq:anshdf} the seller knows the buyer's stopping strategy $\tau$ in advance (which is unreasonable for super-hedging), we may expect that in general it is possible $\overline\pi(\Phi)>\hat\pi(\Phi)$. We shall provide \exref{counter} showing $\overline\pi(\Phi)>\hat\pi(\Phi)$ at the end of this section.

An alternative way to define the super-hedging price is:
\begin{equation}\label{naive}
\tilde\pi(\Phi):=\inf\left\{x\in\mathbb{R}:\ \exists (H,h)\in\mathcal{H}\times\mathbb{R}^e,\ \text{ s.t. } x+(H\cdot S)_T+hg\geq \Phi_\tau,\ \mathcal{P}-q.s.,\ \forall\tau\in\mathcal{T}\right\}.
\end{equation}
However, this definition is not as useful since any reasonable investor would adjust her strategy after observing how the buyer of the option behaves. (In fact, $\mathcal{H}$ can be treated as a subset of $\mathcal{H}'$, and each element in $\mathcal{H}$ is indifferent to stopping strategies used by the buyer, and the non-anticipativity is automatically satisfied.) Due to the fact that for $\tilde\pi$ the seller fails to use the information of the realization of $\tau$, it could very well be the case that $\overline\pi(\Phi)<\tilde\pi(\Phi)$. We shall see in \exref{counter} that it is indeed the case.

If $\mathcal{P}$ is the set of all probability measures on a subset $\Omega'$ of $\Omega$, then under the definition of \eqref{naive}, super-hedging the American option is equivalent to super-hedging the lookback option $\max_{t\leq T}\Phi_t$. To wit, suppose for $x\in\mathbb{R}$ and $(H,h)\in\mathcal{H}\times\mathbb{R}^e$, we have that
\begin{equation}\label{a9}
x+(H\cdot S)_T+hg\geq \Phi_\tau, \ \forall s\in\Omega',\ \forall\tau\in\mathcal{T},
\end{equation}
and
$$x+(H\cdot S)_T+hg<\max_{t\leq T}\Phi_t, \ \text{ along some path } s^*=(s_0^*=1,s_1^*,\dotso,s_T^*)\in\Omega'.$$
Let $t^*=\argmax_{t\leq T}\Phi_t(s^*)$ and define $\tau^*\in\mathcal{T}$ with the property that $\tau(s^*)=t^*$, i.e., the holder of the American option will stop at time $t^*$ once she observes $(s_0^*,\dotso,s_{t^*}^*)$ happens.  Then \eqref{a9} does not hold if we take $\tau=\tau^*$ and $s=s^*$.  So the super-hedging price under the definition of \eqref{naive} is:
$$\tilde\pi(\Phi)=\sup_{Q\in\mathcal{Q}}E_Q\left[\max_{t\leq T}\Phi_t\right].$$

\exref{counter} below shows that it is possible that $\hat\pi(\Phi)<\overline\pi(\Phi)<\tilde\pi(\Phi)$, which indicates that the super-hedging definitions in \eqref{eq:anshdf} and \eqref{naive} are unreasonable.

\begin{example}\label{counter}
We will use the set-up in \exref{counterexample}. An easy calculation shows that
\begin{eqnarray}
\overline\pi(\Phi)&=&\inf_{h\in\mathbb{R}}\sup_{Q\in\mathcal{M}}\sup_{\tau\in\mathcal{T}}E_Q[\Phi_\tau-hg]=\inf_{h\in\mathbb{R}}\sup_{0\leq p, q\leq 1/2}\left[\frac{p}{2}\vee\frac{1}{8}+q\vee\frac{1}{3}-h\left(\frac{p}{2}+\frac{q}{2}-\frac{1}{3}\right)\right]\notag\\
&=&\inf_{h\in\mathbb{R}}\left[\left(\frac{11}{24}+\frac{h}{3}\right)\vee\left(\frac{5}{8}+\frac{h}{12}\right)\vee\left(\frac{7}{12}+\frac{h}{12}\right)\vee\left(\frac{3}{4}-\frac{h}{6}\right)\right]=\frac{2}{3},\notag
\end{eqnarray}
where the infimum is attained when $h=1/2$. On the other hand,
$$\tilde\pi(\Phi)=\sup_{Q\in\mathcal{Q}}E_Q\left[\max_{t\leq T}\Phi_t\right]=\sup_{\substack{0\leq p, q\leq 1/2\\p+q=2/3}}\left(\frac{3}{8}p+\frac{2}{3}q+\frac{11}{24}\right)=\frac{41}{48}>\overline\pi(\Phi),$$
and
$$\hat \pi (\Phi)=\sup_{\tau\in\mathcal{T}}\sup_{Q\in\mathcal{Q}} E_Q[\Phi_\tau]=\sup_{\substack{0\leq p, q\leq 1/2\\p+q=2/3}}\left(\frac{p}{2}\vee\frac{1}{8}+q\vee\frac{1}{3}\right)=\frac{5}{8}<\overline\pi(\Phi).$$
\end{example}

\section{Approximating the hedging-prices by discretizing the path space}

In this section, we take $\mathcal{P}$ to be the set of all the probability measures on $\Omega$ and consider the hedging problems path-wise. We will make the same no-arbitrage assumption and also assume that no hedging option is redundant (see Assumption~\ref{ass:SA}(ii)). We will discretize the path space to obtain a discretized market, and show that the hedging prices in the discretized market converges to the original ones. We also get the rate of convergence. Theorems \ref{main} and \ref{main1} are the main results of this section.

%
 We will now collect some notation that will be used in the rest of this section. The meaning of some of the parameters will become clear when they first appear in context.
  \subsection{Notation}\label{subsec:notation}
\begin{itemize}
\item $\Omega=\{1\}\times[a_1,b_1]\times\dotso\times[a_T,b_T]$, where $0\leq a_T<\dotso<a_1<1<b_1<\dotso<b_T<\infty$. (This means that the wingspan of the discrete-time model is growing as for example it does in a binomial tree market.)
\item $\Omega^n=\Omega\cap\{0,1/2^n,2/2^n,\dotso\}^{T+1}$.
\item $\mathcal{P}$ all the probability measures on $\Omega$.
\item $\mathcal{P}^n$ all the probability measures on $\Omega^n$.
\item $\mathcal{Q}:=\{\mathbb{Q} \text{ martingale measure on } \Omega:\ \mathbb{E}_\mathbb{Q}g_i=0,\ i=1,\dotso,e\}$.
\item $\mathcal{Q}^n:=\{\mathbb{Q} \text{ martingale measure on } \Omega^n:\ \mathbb{E}_\mathbb{Q}g_i=c_i^n,\ i=1,\dotso,e\}$.
\item $\mathcal{H}$ is the set of trading strategies $H=(H_i)_{i=0}^{T-1}$ consists of functions $H_i$ defined on $\prod_{j=1}^i[a_i,b_i],\ i=0,\dotso,T-1$.
\item $\mathcal{H}^n$ is the set of trading strategies $H=(H_i)_{i=0}^{T-1}$ consists of functions $H_i$ defined on $\prod_{j=1}^i[a_j^n,b_j^n]\cap\{0,1/2^n,2/2^n,\dotso\}^i,\ i=0,\dotso,T-1$.
\item $\mathcal{T}$ is the set of stopping times $\tau: \Omega\rightarrow\{0,1,\dotso,T\}$, i.e., for $k=0,1,\dotso,T,\ s^j=(s_0^j,\dotso,s_T^j)\in \Omega,\ j=1,2$,
$$\text{ if }\tau(s^1)=k, \text{ and }s_i^1=s_i^2, \ i=0,\dotso,k, \text{ then } \tau(s^2)=k.$$
\item $\mathcal{T}^n$ is the set of stopping times $\tau: \Omega^n\rightarrow\{0,1,\dotso,T\}.$
\item $\mathcal{H}':=\{H:\mathcal{T}\mapsto \mathcal{H},\text{ s.t. }H_t(\tau^1)=H_t(\tau^2),\ \forall t<\tau^1\wedge\tau^2\}$.
\item ${\mathcal{H}^n}':=\{H:\mathcal{T}^n\mapsto \mathcal{H}^n,\text{ s.t. }H_t(\tau^1)=H_t(\tau^2),\ \forall t<\tau^1\wedge\tau^2\}$.
\item $|\cdot|$ represents the sup norm in various cases.
\item $\mathbb{D}=\cup_n\{0,1/2^n,2/2^n,\dotso\}$.
\end{itemize}

\subsection{Original market}
We restrict the price process, denoted by $S=(S_0,\dotso,S_T)$, to take values in some compact set $\Omega$. In other words,  we take $S$ to be the canonical process $S_i(s_0,\dotso,s_T)=s_i$ for any $(s_0,\dotso,s_T)\in\Omega$, and denote by $\{\mathcal{F}_i\}_{i=1,\dotso,T}$ the natural filtration generated by $S$. The options $(g_i)_{i=1}^e$, which can be bought at price 0, and the American option $\Phi$ are continuous.
 We assume that NA($\mathcal{P}$) holds and that no hedging option is redundant, i.e., it cannot be replicated by the stock and other options available for static hedging. Besides, from the structure of $\Omega$, we know that for $H\in\mathcal{H}$, if $(H\cdot S)_T\geq 0,\ \forall s\in\Omega$, then $H\equiv 0$. Thus, we will make the following standing assumption. 
 
\begin{assumption}\label{ass:SA} (i) $g$ and $\Phi$ are continuous. (ii) For any $(H,h)\in\mathcal{H}\times\mathbb{R}^e$, if $h\neq0$, then there exists $s\in \Omega$, such that along the path $s$, 
$$(H\cdot S)_T+hg<0.$$
\end{assumption}

\begin{example}
Consider the market with $\Omega=\{1\}\times[2/3,4/3]\times[1/3,5/3]$, with a European option $(S_2-1)^+-1/5$ that can be traded at price $0$. A simple calculation can show that Assumption~\ref{ass:SA} is satisfied. 
\end{example}

We consider the sub-hedging price $\underline\pi(\Phi)$ and the super-hedging price $\overline\pi(\Phi)$ with respect to $(\Omega,\mathcal{P})$, i.e.,
\begin{equation}\notag
\underline\pi(\Phi):=\sup\left\{x\in\mathbb{R}:\ \exists (H,\tau,h)\in\mathcal{H}\times\mathcal{T}\times\mathbb{R}^e,\ \text{ s.t. } \Phi_\tau+(H\cdot S)_T+hg\geq x,\ \forall s\in\Omega\right\},
\end{equation}
and
$$\overline\pi(\Phi):=\inf\left\{x\in\mathbb{R}:\ \exists (H,h)\in\mathcal{H}'\times\mathbb{R}^e,\ \text{ s.t. } x+(H\cdot S)_T+hg\geq \Phi_\tau,\ \forall s\in\Omega,\ \forall\tau\in\mathcal{T}\right\}.$$
Recall that $\underline\pi(\Phi)$ and $\overline\pi(\Phi)$ satisfy the dualities in \eqref{dual1} and \eqref{supdual} respectively.

\subsection{Discretized market}

For simplicity, we assume that $a_i,\ b_i\in\mathbb{D},\ i=1,\dotso,T$, in the notation of $\Omega$, and we always start from $n$ large enough, such that $\Omega^n$ has the end points $a_i,b_i$ at each time $i$. Let $\{c^n=(c_1^n,\dotso,c_e^n)\}_n$ be a sequence such that $|c^n|\rightarrow 0$.  Now for each $n$, consider the following discretized market:  The stock price process takes values in the path space $\Omega^n$, and the options $(g_i)_{i=1}^e$ can be traded at the beginning at price $(c_i^n)_{i=1}^e$. 

We consider the sub-hedging price $\underline\pi^n(\Phi)$ and the super-hedging price $\overline\pi^n(\Phi)$ with respect to $(\Omega^n,\mathcal{P}^n)$, i.e.,
\begin{equation}\notag
\underline\pi^n(\Phi):=\sup\left\{x\in\mathbb{R}:\ \exists (H,\tau,h)\in\mathcal{H}^n\times\mathcal{T}^n\times\mathbb{R}^e,\ \text{ s.t. } \Phi_\tau+(H\cdot S)_T+hg\geq x,\ \forall s\in\Omega^n\right\},
\end{equation}
and
$$\overline\pi^n(\Phi):=\inf\left\{x\in\mathbb{R}:\ \exists (H,h)\in{\mathcal{H}^n}'\times\mathbb{R}^e,\ \text{ s.t. } x+(H\cdot S)_T+hg\geq \Phi_\tau,\ \forall s\in\Omega^n,\ \forall\tau\in\mathcal{T}^n\right\}.$$
Recall that $\underline\pi^n(\Phi)$ and $\overline\pi^n(\Phi)$ satisfy the dualities in \eqref{dual1} and \eqref{supdual} respectively.

\begin{remark}
Assuming $a_i, b_i\in\mathbb{D}$ and the points in $\Omega^n$ is equally spaced is without loss of generality. In fact, as long as $\Omega^n\cap \Omega$ are increasing and $\overline{\cup_n(\Omega^n\cap \Omega)}=\Omega$, we will have the same results with only a little adjustment in the proofs.
\end{remark}

\subsection{Consistency}\label{con}
The following theorem states that for $n$ large enough, the discretized market is well defined, i.e., NA$(\mathcal{P}^n)$ holds.  
\begin{theorem}\label{consist}
For $n$ large enough, NA$(\mathcal{P}^n)$ holds.
\end{theorem}
\begin{proof}
If not, then there exists $(H^n,h^n)\in\mathcal{H}^n\times\mathbb{R}^e$, such that
\begin{equation}\label{a2}
(H^n\cdot S)_T+h^n(g-c^n)\geq 0,\ \ \ \forall s\in\Omega^n,
\end{equation}
and is strictly positive along some path in $\Omega^n$. Obviously, $h^n\neq 0$, so without loss of generality we will assume that $|h^n|=1$. On the other hand, since $g$ is continuous on a compact set it is bounded.  Then there exists a constant $C>0$ independent of $n$, such that
\begin{equation}\label{a3}
(H^n\cdot S)_T>-C.
\end{equation}

We will need the following result in order to carry out the proof of the theorem. We preferred to separate this result from the proof of the theorem since it will be used again in the proof of the convergence result.

\begin{lemma}\label{lemma 4.1}
If $(H^n\cdot S)_T>-C$, then there exists a constant $M=M(C)>0$ independent of $n$, such that $|H^n| \leq M$.
\end{lemma}
\begin{proof}
Let $\alpha:=\min_{1\leq i\leq T}\{a_{i-1}-a_i,b_i-b_{i-1}\}>0$, with $a_0:=b_0:=1$. We will prove this by an induction argument. Take the path $(s_0=1,s_1=a_1,s_2=a_1,\dotso, s_T=a_1)$, then \eqref{a3} becomes
$$H_0^n(a_1-1)>-C,$$
which implies $H_0^n<C/\alpha$. Similarly, we can show that $H_0^n>-C/\alpha$ by taking the path $(s_0=1,s_1=b_1,s_2=b_1,\dotso, s_T=b_1)$. Hence, $H_0^n$ is bounded uniformly in $n$. Now assume there exists $K=K(C)>0$ independent of $n$, such that $|H_j^n|\leq K,\ j\leq i-1\leq T-1$. Since $\Omega^n$ is uniformly bounded and by the induction hypothesis, we have that
$$\sum_{j=i}^{T-1}H_j^n(s_1,\dotso,s_j)(s_{j+1}-s_j)>-C',$$
where $C'>0$ only depends on $C$. For any $(s_1,\dotso,s_i)\in\prod_{j=1}^i([a_j,b_j]\cap\{k/2^n,\ k\in\mathbb{N}\})$, by taking the paths $(1,s_1,\dotso,s_i, s_{i+1}=a_{i+1},\dotso, s_T=a_{i+1})$ and $(1,s_1,\dotso,s_i, s_{i+1}=b_{i+1},\dotso, s_T=b_{i+1})$, we can show that $|H_i^n(s_1,\dotso,s_i)|\leq C'/\alpha$.
\end{proof}

\noindent \emph{Proof of Theorem~\ref{consist} continued.}
We proved in Lemma~\ref{lemma 4.1}  that $|H^n|\leq M$ for some $M>0$ independent of $n$. By a standard selection (using a diagonalization argument, e.g., see \cite[Page 307]{Resnick}), we can show that there exists a subsequence $(H^{n_k},h^{n_k})\stackrel{|\cdot|}{\rightarrow}(H,h)$, where $H=(H_i)_{i=0}^{T-1}$ consists of functions $H_i$ defined on $\prod_{j=1}^i([a_j,b_j]\cap\mathbb{D}),\ i=0,\dotso,T-1,$ with $|H|\leq M$, and $h\in\mathbb{R}^e$ with $|h|=1$. By taking the limit on both sides of \eqref{a2} along $(n_k)$, we have
\begin{equation}\label{dense}
(H\cdot S)_T+hg\geq0,\ \ \ \forall s\in\Omega\cap\mathbb{D}^{T+1}.
\end{equation}
If we can extend the domain of function $H$ from $\Omega\cap\mathbb{D}^{T+1}$ to $\Omega$, such that the inequality \eqref{dense} still holds on $\Omega$, we would obtain a contradiction to Assumption~\ref{ass:SA} since $h\neq 0$. 

Define 
$$\tilde\Omega_i=\{1\}\times[a_1,b_1]\times\dotso\times[a_i,b_i]\times\big([a_{i+1},b_{i+1}]\cap\mathbb{D}\big)\times\dotso\times\big([a_T,b_T]\cap\mathbb{D}\big)$$
for $i=1,\dotso,T-1$. We will do the extension inductively as follows (the notation for $H$ will not be changed during the extension):\\
(i) For each $s_1\in[a_1,b_1]\setminus\mathbb{D}$, using the standard selection argument, we can choose $[a_1,b_1]\cap\mathbb{D}\ni s_1^n\rightarrow s_1$, such that for any $j\in\{1,\dotso,T-1\}$ and $(s_2,\dotso,s_j)\in\prod_{k=2}^j\big([a_k,b_k]\cap\mathbb{D}\big)$, the limit $\lim_{n\rightarrow\infty}H(s_1^n,s_2,\dotso,s_j)$ exists. Define 
$$H_j(s_1,\dotso,s_j):=\lim_{n\rightarrow\infty}H_j(s_1^n,s_2,\dotso,s_j).$$
Then we extend the domain of $H$ to $\tilde\Omega_1$. It's easy to check that \eqref{dense} still holds on $\tilde\Omega_1$.\\
(ii) In general, assume that we have already extended the domain of $H$ to $\tilde\Omega_i,\ i\leq T-2$, such that \eqref{dense} holds on it. Then for each $(s_1,\dotso,s_i)\in\prod_{j=1}^i[a_j,b_j]$ and $s_{i+1}\in[a_{i+1},b_{i+1}]\setminus\mathbb{D}$, performing the same selection and extension as in (i) (we fix $(s_1,\dotso,s_i)$ while doing the selection), we can see that \eqref{dense} still holds on $\tilde\Omega_{i+1}$.\\
Therefore, we can extend $H$ to $\tilde\Omega_{T-1}$, such that \eqref{dense} holds. Clearly, \eqref{dense} also holds on $\Omega$.
\end{proof}

\subsection{Convergence} We shall prove the convergence result for sub-hedging (\thref{main}). The super-hedging case is similar, and thus we shall only provide the corresponding result (\thref{main1}) without proof.
\begin{lemma}\label{lemma 4.2}
For $(H^n,\tau^n,h^n)\in\mathcal{H}^n\times\mathcal{T}^n\times\mathbb{R}^e$, if for $x \in \mathbb{R}$
\begin{equation}\label{bd}
\Phi_{\tau^n}+(H^n\cdot S)_T+h^n(g-c^n)\geq x,\ \ \ \forall s\in\Omega^n,
\end{equation}
then $(H^n)_n$ and $(h^n)_n$ are bounded. 
\end{lemma}
\begin{proof}
We first show that $(h^n)_n$ are bounded. If not, by extracting a subsequence, we can without loss of generality assume that $0<\beta<|h^n|\rightarrow\infty$. We consider two cases:\\
(a) $|H^n|/|h^n|$ is not bounded. Then we can rewrite \eqref{bd} as
$$\left(\frac{H^n}{|h^n|}\cdot S\right)_T\geq-\frac{h^n}{|h^n|}(g-c^n)+\frac{1}{|h^n|}\Phi_{\tau^n}+\frac{x}{|h^n|},\ \forall s\in\Omega^n.$$
Since $g$ and $\Phi$ are continuous on a compact set, they are bounded. Hence, there exists $C>0$, such that
$$\left(\frac{H^n}{|h^n|}\cdot S\right)_T\geq-C,$$
which contradicts with \leref{lemma 4.1}.\\
(b) $|H^n|/|h^n|$ is bounded. Let us rewrite \eqref{bd} as
$$\left(\frac{H^n}{|h^n|}\cdot S\right)_T+\frac{h^n}{|h^n|}(g-c^n)\geq\frac{x+\Phi_{\tau^n}}{|h^n|},\ \forall s\in\Omega^n.$$
Since $(x+\Phi_{\tau^n})/|h^n|\rightarrow 0$, we can follow the proof of \thref{consist} to get a contradiction with Assumption~\ref{ass:SA}.

Next we show that $(H^n)_n$ is a bounded collection. Let us rewrite \eqref{bd} as 
$$(H^n\cdot S)_T\geq-\Phi_{\tau^n}-h^n(g-c^n)+x,\ \ \ \forall s\in\Omega^n.$$
Since $(h^n)_n$ and $(g-c^n)_n$ are bounded, then right-hand-side is bounded. Therefore, the conclusion follows from \leref{lemma 4.1}. 
\end{proof}
\begin{proposition}
For $n$ large enough, there exists some $N>0$ independent of $n$, such that
\begin{equation}\label{subNn}
\underline\pi^n(\Phi)=\sup\left\{x\in\mathbb{R}:\ \exists (H,\tau,h)\in\mathcal{H}^n\times\mathcal{T}^n\times\mathbb{R}^e,\ |H|,|h|\leq N,\ \text{ s.t. } \Phi_\tau+(H\cdot S)_T+hg\geq x,\ \forall s\in\Omega^n\right\}.
\end{equation}
and
\begin{equation}\label{subN}
\underline\pi(\Phi)=\sup\left\{x\in\mathbb{R}:\ \exists (H,\tau,h)\in\mathcal{H}\times\mathcal{T}\times\mathbb{R}^e,\ |H|,|h|\leq N,\ \text{ s.t. } \Phi_\tau+(H\cdot S)_T+hg\geq x,\ \forall s\in\Omega\right\}.
\end{equation}
\end{proposition}
\begin{proof}
Let $\underline x:=\min_{(t,s)\in\{1,\dotso,T\}\times\Omega}\Phi(t,s)$. It is easy to see that
\begin{equation}\label{eq:wlbdd}
\underline\pi^n(\Phi)=\sup\left\{x\geq\underline x:\ \exists (H,\tau,h)\in\mathcal{H}^n\times\mathcal{T}^n\times\mathbb{R}^e,\ \text{ s.t. } \Phi_\tau+(H\cdot S)_T+hg\geq x,\ \forall s\in\Omega^n\right\}.
\end{equation}
For $n$ large enough, the set 
$$\{(H^n,h^n)\in\mathcal{H}^n\times\mathbb{R}^e:\exists\tau\in\mathcal{T}^n,\text{ s.t. }\Phi_\tau+(H\cdot S)_T+hg\geq \underline x,\ \forall s\in\Omega^n\}$$
is uniformly bounded in $n$, which is indicated by \leref{lemma 4.2}. Since this set of strategies is the largest among the ones we need to consider for sub-hedging, thanks to \eqref{eq:wlbdd},  there exists a constant $N>0$, such that for $n$ large enough,
$$\underline\pi^n(\Phi)=\sup\left\{x\geq\underline x:\ \exists (H,\tau,h)\in\mathcal{H}\times\mathcal{T}\times\mathbb{R}^e,\ |H^n|,|h^n|\leq N\ \text{ s.t. } \Phi_\tau+(H\cdot S)_T+hg\geq x,\ \forall s\in\Omega^n\right\},$$
which implies \eqref{subNn}.

Similarly, we have that the set 
$$\{(H,h)\in\mathcal{H}\times\mathbb{R}^e:\exists\tau\in\mathcal{T},\text{ s.t. }\Phi_\tau+(H\cdot S)_T+hg\geq \underline x,\ \forall s\in\Omega\}$$
is bounded. Otherwise, there exists $(H^m,\tau^m,h^m)\in\mathcal{H}\times\mathcal{T}\times\mathbb{R}^e$, such that
$$\Phi_{\tau^m}+(H^m\cdot S)_T+h^mg\geq \underline x,\ \forall s\in\Omega\cap\mathbb{D}^{T+1},$$
with $|H^m|+|h^m|\rightarrow\infty$. Then we can use a similar argument to the one in the proof of \thref{consist} to get a contradiction. Now \eqref{subN} follows.
\end{proof}

\begin{theorem}\label{main}
Under Assumption~\ref{ass:SA}, we have
\begin{equation}\label{a6}
\lim_{n\rightarrow\infty}\underline\pi^n(\Phi)=\underline\pi(\Phi).
\end{equation}
Furthermore, if $\Phi$ and $g$ are Lipschitz continuous, then
\begin{equation}\label{a7}
|\underline\pi^n(\Phi)-\underline\pi(\Phi)|=O(1/2^n)
\end{equation}
by taking $|c^n|=O(1/2^n)$.
\end{theorem}
\begin{proof}
For $x\in(\underline\pi(\Phi)-\eps,\underline\pi(\Phi)]$, there exists $(H,\tau,h)\in\mathcal{H}\times\mathcal{T}\times\mathbb{R}^e$,with $|H|,|h|\leq N$, such that
$$\Phi_\tau+(H\cdot S)_T+hg\geq x,\ \forall s\in\Omega.$$
Hence,
$$\Phi_\tau+(H\cdot S)_T+h(g-c^n)\geq x-eN|c^n|,\ \forall s\in\Omega^n.$$
Therefore, 
$$\underline\pi(\Phi)-\eps-eN|c^n|\leq x-eN|c^n|\leq \underline\pi^n(\Phi).$$
By letting $\eps\rightarrow 0$, we have
\begin{equation}\label{a4}
\underline\pi^n(\Phi)\geq\underline\pi(\Phi)-eN|c^n|.
\end{equation}
On the other hand, for $x^n\in(\underline\pi^n(\Phi)-\eps,\underline\pi^n(\Phi)]$, there exists $(H^n,\tau^n,h^n)\in\mathcal{H}^n\times\mathcal{T}^n\times\mathbb{R}^e$,with $|H^n|,|h^n|\leq N$, such that
\begin{equation}
\label{a5}
\Phi_{\tau^n}+(H^n\cdot S)_T+h^n(g-c^n)\geq x^n,\ \forall s\in\Omega^n.
\end{equation}
Consider the map $\phi^n: \Omega\rightarrow\Omega^n$ given by
$$\phi^n(1,s_1,\dotso,s_T)=(1,\lfloor 2^n s_1\rfloor/2^n,\dotso,\lfloor 2^n s_T\rfloor/2^n),\ \ \ \forall(1,s_1,\dotso,s_T)\in\Omega.$$ 
Also define $(H,\tau)\in\mathcal{H}\times\mathcal{T}$ as
\begin{equation}\label{eps}
H(s)=H^n(\phi^n(s))\  \text{ and }\tau(s)=\tau^n(\phi^n(s))
\end{equation}
 Since $\Phi$ and $g$ are continuous on a compact set, they are uniformly continuous. Also $(H^n, q^n)_n$ are uniformly bounded, and $c^n\rightarrow 0$.  Then from \eqref{a5} we have that for $n$ large enough, the trading strategy $(H,\tau)$ defined in \eqref{eps} satisfies
\begin{equation}\label{a8}
\Phi_{\tau}+(H\cdot S)_T+h^{n}g\geq x^{n}-\eps,\ \forall s\in\Omega,
\end{equation}
by noting that $\phi^n(s)\rightarrow s$ uniformly and $\tau(s)=\tau(\phi^n(s))$. Thus, $\underline\pi(\Phi)>\underline\pi^n(\Phi)-2\eps$. Combining with \eqref{a4}, we have \eqref{a6}.

If $\Phi$ and $g$ are Lipschitz continuous, then we have a stronger version of \eqref{a8}:
\begin{equation}\notag
\Phi_{\tau}+(H\cdot S)_T+h^{n}g\geq x^{n}-eN|c^n|-C/2^n,\ \forall s\in \Omega,
\end{equation}
where $C>0$ is a constant only depends on $N, e,T$ and the Lipschitz constants of $\Phi$ and $g$. Hence,
$$\underline\pi^n(\Phi)-\eps-eN|c^n|-C/2^n\leq x^{n}-eN|c^n|-C/2^n\leq\underline\pi(\Phi).$$
Letting $\eps\rightarrow 0$ and taking $|c^n|=O(1/2^n)$, and combining with \eqref{a4}, we obtain \eqref{a7}.
\end{proof}

Similar to the proof of the sub-hedging case, we can show the following convergence result for super-hedging.
\begin{theorem}\label{main1}
Under Assumption~\ref{ass:SA}, we have
\begin{equation}\label{a6}
\lim_{n\rightarrow\infty}\overline\pi^n(\Phi)=\overline\pi(\Phi).
\end{equation}
Furthermore, if $\Phi$ and $g$ are Lipschitz continuous, then
\begin{equation}\label{a7}
|\overline\pi^n(\Phi)-\overline\pi(\Phi)|=O(1/2^n)
\end{equation}
by taking $|c^n|=O(1/2^n)$.
\end{theorem}

\subsection{A suitable construction for $c^n$ and $\mathcal{Q}^n$} \label{sec:suitable} In Section \ref{con} we obtained that as long as $c^n\rightarrow 0$, then for $n$ large enough, NA$(\mathcal{P}^n)$ holds, which implies $\mathcal{Q}^n\neq\emptyset$ (see \cite[Theorem 1.3]{Schachermayer2} or \cite[FTAP]{Nutz2}). The theorem below gives a more specific way to construct $c^n$, such that $\mathcal{Q}^n\neq\emptyset$ for all $n$ with $\Omega^n\subset\Omega$, when all the hedging options are vanilla. [This analysis would be useful for the consistency, when there are infinitely many options and the marginal distribution of the stock price (at the maturities of the hedging European options) under the martingale measures appearing in the duality are fixed.]

\begin{proposition}
Let $\mu_0,\dotso,\mu_T$ be the marginal of a martingale measure on $\mathbb{R}_+^{T+1}$. Then there exist a collection of probability measures $\{\mu_i^n: i=0,\dotso,T,\ n\in\N\}$ on $\mathbb{R}$ such that
\begin{itemize}
\item [(1)] $\mu_i^n\stackrel{w}{\rightarrow}\mu_i,\ i=0,\dotso,T$,
\item [(2)] $\mu_i^n(K^n)=1,\ i=0,\dotso,T$,
\item [(3)] For each $n\in\N$, $\mathcal{M}^n\neq\emptyset$,
\end{itemize}
where $K^n=\{0,1/2^n,2/2^n,\dotso\}$ and $\mathcal{M}^n$ is the set of martingale measures on $(K^n)^{T+1}$ with marginals $(\mu_i^n)_{i=0}^T$.
\end{proposition}
\begin{proof}
Fix $i\in\{0,\cdots,T\}$. For any $n\in\N$, define a measure $\mu^n_i$ on $\{0,1/2^n,2/2^n,\cdots\}$ by
\begin{align*}
\mu^n_i(\{0\})&:=\int_0^{1/2^n}(1-2^nx)d\mu_i(x),\\
\mu^n_i(\{k/2^n\})&:=\int_{(k-1)/2^n}^{k/2^n}(2^nx+1-k)d\mu(x)+\int_{k/2^n}^{(k+1)/2^n}(1+k-2^nx)d\mu(x),\ \ \forall k\in\N.
\end{align*}
By construction, we have $\sum_{k\in\N\cup\{0\}}\mu^n_i(\{k/2^n\}) =\int_{\R_+}d\mu_i(x)=1$. It follows that $\mu^n_i$ is a probability measure on $\{0,1/2^n,2/2^n,\cdots\}$. 

For any function $h:\R\mapsto\R$, consider the piecewise linear function $h^n$ defined by setting $h^n(k/2^n):=h(k/2^n)$ for $k\in\N\cup\{0\}$. We define $h^n(x)$ for $x\in\R_+\setminus\{0,1/2^n,2/2^n,\cdots\}$ using linear interpolation. That is,  for any $x\in \mathbb{R}_+$,
\begin{align*}
h^n(x)&:=(1+\lfloor 2^nx\rfloor-2^nx)h\left(\frac{\lfloor 2^nx\rfloor}{2^n}\right)+(2^nx-\lfloor 2^nx\rfloor)h\left(\frac{1+\lfloor 2^nx\rfloor}{2^n}\right)\\
&=h\left(\frac{k}{2^n}\right)(1+k-2^nx)+h\left(\frac{k+1}{2^n}\right)(2^nx-k),\ \ \forall k\in\N\cup\{0\}.
\end{align*}
From the above identity and the definition of $\mu^n_i$, we observe that 
\begin{equation}\label{h=h^n}
\int_{\R_+} h d\mu^n_i=\int_{\R_+}h^n d\mu_i.
\end{equation}
Now, if we take $h$ to be an arbitrary bounded continuous function, then $h^n\to h$ pointwise and the integrals in \eqref{h=h^n} are finite. By using \eqref{h=h^n} and the dominated convergence theorem, we have $\int_{\R_+} h d\mu^n_i\to\int_{\R_+} h d\mu_i$. This shows that $\mu_i^n\stackrel{w}{\rightarrow}\mu_i$. On the other hand, if we take $h$ to be an arbitrary convex function, then $h^n$ by definition is also convex. Thanks to \cite[Theorem 8]{Strassen65}, the convexity of $h^n$ imply that $\int_{\R_+}h^n d\mu_i$ is nondecreasing in $i$. We then obtain form \eqref{h=h^n} that $\int_{\R_+} h d\mu^n_i$ is nondecreasing in $i$. Since this holds for any given convex function $h$, we conclude from \cite[Theorem 8]{Strassen65} that $\mathcal{M}^n\neq\emptyset$.
\end{proof}

Now we further assume that the finitely many options are vanilla. Take $Q\in\mathcal{Q}$ and let $\mu_i$ be the distribution of $S_i$ under $Q$ for $i=1,\dotso,T$. From the theorem above (and the construction of $\mu_i^n$), there exists a martingale measures $Q^n$ supported on $\Omega^n$, with marginals $\mu_i^n\stackrel{w}{\rightarrow}\mu_i$, for $i=1,\dotso,T$. Set 
$$c_i^n:=\mathbb{E}_{Q^n}[g_i]-\mathbb{E}_Q[g_i],\ \ i=1,\dotso,e.$$
Then, we have $c^n\rightarrow 0$ by the weak convergence of the marginals, and $\mathcal{Q}^n\neq\emptyset$  for all $n$ with $\Omega^n\subset\Omega$,  since $Q^n\in\mathcal{Q}^n$. In addition, if $g$ is Lipschitz continuous, we have that $|c^n|=O(1/2^n)$.

\begin{appendices}
\begin{center}
\vspace{1cm}

APPENDIX
\end{center}
\section{Proof of \prref{prop 1.1}}\label{sec:appA}
\begin{proof}[Proof of \prref{prop 1.1}]
Following the proof of \cite[Lemma 5.3]{ZZ}, it can be shown that for $t\in\{0,\dotso,T-1\}$ and $\omega\in\Omega_{T-t}$,
$$\mathcal{M}_t(\omega)=\{Q\in\kP(\Omega_{T-t}):\ Q\ll P \text{ for some }P\in\P_t(\omega),\ (S_k(\omega,\cdot))_{k=t,\dotso,T}\text{ is a }Q\text{-martingale}\}.$$
Hence, in order to show the analyticity of graph($\cM_t$), it suffices to show that the sets
$$\cI:=\{(\omega,\Q)\in\Omega_t\times\kP(\Omega_{T-t}):\ Q\ll P \text{ for some }P\in\P_t(\omega)\}$$
and
$$\cJ:=\{(\omega,\Q)\in\Omega_t\times\kP(\Omega_{T-t}):\ (S_k(\omega,\cdot))_{k=t,\dotso,T}\text{ is a }Q\text{-martingale}\}$$
are analytic.

Thanks to the analyticity of graph($\P_t$), we can follow the argument in the proof of \cite[Lemma 4.8]{Nutz2} to show that $\cI$ is analytic. Now let us consider $\cJ$. For $k=t,\dotso,T-1$, there exists a countable algebra $(A_i^k)_{i=1}^\infty$ generating $\mathcal{F}_k$. Then
$$\cI=\bigcap_{k=t}^{T-1}\bigcap_{i=1}^\infty\{(\omega,\Q)\in\Omega_t\times\kP(\Omega_{T-t}):\ E_Q[\Delta S_k(\omega,\cdot)1_{A_i^k}(\omega,\cdot)]=0\}.$$
By a monotone class argument, we can show that for $(\omega,\Q)\in\Omega_t\times\kP(\Omega_{T-t})$, the map
$$(\omega,Q)\mapsto E_Q[\Delta S_k(\omega,\cdot)1_{A_i^k}(\omega,\cdot)]$$
is Borel measurable (e.g., see the first paragraph in the proof of \cite[Theorem 2.3]{Nutz3}). Therefore, the set $\cJ$ is Borel measurable, and in particular it is analytic.
\end{proof}

\section{Optimal Stopping for Adverse Nonlinear Expectations}\label{sec:appB}
In this section, we analyze both the adverse optimal stopping problems for nonlinear expectations. This result is used in Theorem~\ref{sub} for showing the existence of the sub hedging strategy. Note that  \cite{2013arXiv1301.0091B,2012arXiv1209.6601E,2012arXiv1212.2140N} analyze similar problems in continuous time. Instead of referring to these papers directly, we decided to include a short analysis here because it is much simpler to carry it out  in discrete time using backward induction.

For each $t\in\{0,\dotso,T-1\}$ and $\omega\in\Omega_t$, we are given a nonempty convex set $\mathcal{R}_t(\omega)\subset\mathfrak{P}(\Omega_1)$ of probability measures. We assume that for each $t$, the graph of $\mathcal{R}_t$ is analytic, and thus admits a universally measurably selector. For $t=0,\dotso,T-1$ and $\omega\in\Omega_t$, define 
$$\mathfrak{R}_t(\omega):=\{P_t\otimes\dotso\otimes P_{T-1}:\ P_i(\omega,\cdot)\in\mathcal{R}_i(\omega,\cdot),\ i=t,\dotso,T-1\},$$
where each $P_i$ is a universally measurable selector of $\mathcal{R}_i$. We write $\mathcal{R}$ for $\mathfrak{R}_0$ for short. We assume the graph of $\mathfrak{R}_t$ is analytic for $t=0,\dotso,T-1$. Let $\xi$ be a u.s.a. function. For $\omega\in\Omega$, define the nonlinear conditional expectation as
$$\mathcal{E}_t[\xi](\omega)=\sup_{P\in\mathfrak{R}_t(\omega^t)}E_P[\xi(\omega^t,\cdot)].$$
We also write $\mathcal{E}$ for $\mathcal{E}_0$ for short. By \cite[Theorem 2.3]{Nutz3}, we know that the function $\mathcal{E}_t[\xi]$ is u.s.a. and $\mathcal{F}_t$-measurable, and the nonlinear conditional expectation satisfies the tower property, i.e., for $0\leq s<t\leq T$, it holds that
\begin{equation}\label{tower}
\mathcal{E}_s\mathcal{E}_t[\xi]=\mathcal{E}_s[\xi].
\end{equation}
Moreover, by Galmarino's test (see \cite[Lemma 2.5]{Nutz3}), it follows that if a function is $\mathcal{F}_t$-measurable, it only depends on the path up to time $t$. Throughout this section, we will assume that $f$ is an adapted process with respect to the raw filtration $(\mathcal{B}(\Omega_t))_{t=0}^T$.

We consider the optimal stopping problem
\begin{equation}\label{ap1}
X:=\inf_{\tau\in\mathcal{T}}\mathcal{E}[f_\tau].
\end{equation}
and define the upper value process
\begin{equation}\label{ap2}
X_t:=\inf_{\tau\in\mathcal{T}_t}\mathcal{E}_t[f_\tau],
\end{equation}
and the lower value process
\begin{equation}\label{bb1}
Y_t(\omega):=\sup_{P\in\mathfrak{R}_t(\omega^t)}\inf_{\tau\in\mathcal{T}_t}E_P[f_\tau(\omega^t,\cdot)].
\end{equation}
In particular $X=X_0$. We have the following result:
\begin{theorem}\label{infsup}
Assume for $t=1,\dotso,T-1$, $\mathcal{E}_t[X_{t+1}]$ (or $\mathcal{E}_t[Y_{t+1}]$) is $\mathcal{B}(\Omega_t)$-measurable. Then $X_t=Y_t,\ t=0,\dotso,T$. In particular, the game defined in \eqref{ap1} has a value, i.e.,
\begin{equation}\label{ap3}
\inf_{\tau\in\mathcal{T}}\mathcal{E}[f_\tau]=\sup_{P\in\mathcal{R}}\inf_{\tau\in\mathcal{T}}{E}[f_\tau].
\end{equation}
Moreover, there exists an optimal stoping time described by
\begin{equation}\label{ap4}
\tau^*=\inf\{t\geq 0:\ f_t=X_t\}.
\end{equation}
\end{theorem}
\begin{proof}
We shall prove the result under the Borel measurability assumption for $\mathcal{E}_t[X_{t+1}]$. In fact, it could be seen from the proof later on that the Borel measurability assumption on $\mathcal{E}_t[X_{t+1}]$ is equivalent to that on $\mathcal{E}_t[Y_{t+1}]$.\\
\textbf{Step 1:} We first show that for $s\in\{0,\dotso,T\}$,
\begin{equation}\label{ap5}
X_s=\inf_{\tau\in\mathcal{T}_s}\mathcal{E}_s(f_\tau1_{\{\tau<t\}}+X_t1_{\{\tau\geq t\}}),\quad 0\leq s<t\leq T.
\end{equation}
We shall prove it by a backward induction. For $s=T-1$, since $\tau$ equals either $T-1$ or $T$, we have from \eqref{ap2} that $X_{T-1}=f_{T-1}\wedge\mathcal{E}_{T-1}(f_T)=f_{T-1}\wedge\mathcal{E}_{T-1}(X_T)$, and thus \eqref{ap5} holds. Assume for $s+1\in\{0,\dotso,T-1\}$ the corresponding conclusion holds. Let $t\in\{s+1,\dotso,T\}$. For any $\tau\in\mathcal{T}_s$, using the tower property \eqref{tower} and the definition of $X_t$ in \eqref{ap2}, we have that
$$\mathcal{E}_s(f_\tau)=\mathcal{E}_s\left(f_\tau1_{\{\tau<t\}}+\mathcal{E}_t(f_{\tau\vee t})1_{\{\tau\geq t\}}\right)\geq\mathcal{E}_s\left(f_\tau1_{\{\tau<t\}}+X_t1_{\{\tau\geq t\}}\right),$$
which implies the inequality \lq\lq$\geq$\rq\rq\  in \eqref{ap5}. 

Let us turn to the inequality \lq\lq$\leq$\rq\rq\  in \eqref{ap5}. By the induction assumption, we have that for $k\geq s+1$,
\begin{equation}\label{ap7}
X_k=\inf_{\tau\in\mathcal{T}_k}\mathcal{E}_k(f_\tau1_{\{\tau<k+1\}}+X_{k+1}1_{\{\tau\geq k+1\}})=f_k\wedge\mathcal{E}_k(X_{k+1}).
\end{equation}
Define
\begin{eqnarray}
A_s&:=&\{f_s\leq\mathcal{E}_s(X_{s+1})\}\ \in\ \mathcal{B}(\Omega_s),\notag\\
A_k&:=&\left[\{f_k\leq\mathcal{E}_k(X_{k+1})\}\setminus(\cup_{i=s}^{k-1}A_i)\right]=\left[\{f_k=X_k\}\setminus(\cup_{i=s}^{k-1}A_i)\right]\ \in\ \mathcal{B}(\Omega_k),\quad k=s+1,\dotso,T.\notag
\end{eqnarray}
Note that $A_T=(\cup_{i=s}^{T-1}A_i)^c\in\mathcal{B}(\Omega_{T-1})$. Denoting 
\begin{equation}\label{ap6}
\bar\tau=\sum_{k=s}^Tk1_{A_k}\in\mathcal{T}_s.
\end{equation}
and using the tower property repeatedly, we obtain that
\begin{eqnarray}
X_s&\leq&\mathcal{E}_s(f_{\bar\tau})\notag\\
&=&\mathcal{E}_s\left(\sum_{k=s}^{T-2}f_k1_{A_k}+f_{T-1}1_{A_{T-1}}+\mathcal{E}_{T-1}(X_T)1_{(\cup_{i=s}^{T-1}A_i)^c}\right)\notag\\
&=&\mathcal{E}_s\left(\sum_{k=s}^{T-2}f_k1_{A_k}+X_{T-1}1_{(\cup_{i=s}^{T-2}A_i)^c}\right)\notag\\
&=&\mathcal{E}_s\left(\sum_{k=s}^{T-3}f_k1_{A_k}+f_{T-2}1_{A_{T-2}}+\mathcal{E}_{T-2}(X_{T-1})1_{(\cap_{k=s}^{T-2}A_k)^c}\right)\notag\\
&=&\dotso\notag\\
&=&\mathcal{E}_s\left(f_s1_{A_s}+X_{s+1}1_{A_s^c}\right)\notag\\
\label{s1}&=&f_s\wedge\mathcal{E}_s(X_{s+1}).
\end{eqnarray}
On the other hand, for $t\in\{s+1,\dotso,T\}$, by \eqref{ap7} and the tower property, we have that 
\begin{eqnarray}
X_s&\geq&\inf_{\tau\in\mathcal{T}_s}\mathcal{E}_s\left(f_\tau1_{\{\tau<t\}}+X_t1_{\{\tau\geq t\}}\right)\notag\\
&\geq&\inf_{\tau\in\mathcal{T}_s}\mathcal{E}_s\left(f_\tau 1_{\{\tau<t-1\}}+X_{t-1} 1_{\{\tau=t-1\}}+\mathcal{E}_{t-1}(X_t)1_{\{\tau\geq t\}}\right)\notag\\
&\geq&\inf_{\tau\in\mathcal{T}_s}\mathcal{E}_s\left(f_\tau1_{\{\tau<t-1\}}+X_{t-1}1_{\{\tau\geq t-1\}}\right)\notag\\
&\geq&\dotso\notag\\
&\geq&\inf_{\tau\in\mathcal{T}_s}\mathcal{E}_s\left(f_\tau1_{\{\tau<s+1\}}+X_{s+1}1_{\{\tau\geq s+1\}}\right)\notag\\
\label{s2}&=&f_s\wedge\mathcal{E}_s(X_{s+1}).
\end{eqnarray}
Hence, we have \eqref{ap5} holds for $s$.

\noindent\textbf{Step 2:} Define $\hat\tau=\sum_{k=0}^Tk1_{A_k}$, same as $\bar\tau$ defined in \eqref{ap6} for $s=0$. From \eqref{s1} \& \eqref{s2} in Step 1, we have that $X=\mathcal{E}(f_{\hat\tau})$. Noting $A_0=\{f_0\leq\mathcal{E}(X_1)\}=\{f_0=X\}$, we have $\hat\tau=\tau^*$.

\noindent\textbf{Step 3:} Using \eqref{ap5}, we can follow the proof of \cite[Lemma 4.11]{2012arXiv1212.2140N} mutatis mutandis, to show by a backward induction that $X_t=Y_t,\ t=0,\dotso,T$. In particular \eqref{ap3} holds.
\end{proof}

The next remark is concerned with the ``sup sup" version of the optimal stopping problem:
\begin{remark}
For the optimal stopping problem 
$$Z:=\sup_{\tau\in\mathcal{T}}\mathcal{E}[f_\tau],$$
let us define 
$$Z_t:=\sup_{\tau\in\mathcal{T}_t}\mathcal{E}_t[f_\tau],\quad t=0,\dotso,T.$$
In particular $Z=Z_0$. Following Steps 1 and 2 in the proof of \thref{infsup}, we can show that if
$\mathcal{E}_t[Z_{t+1}]$ is $\mathcal{B}(\Omega_t)$-measurable for $t=1,\dotso,T-1$, then 
\begin{equation}\notag
Z_t=f_t\vee\mathcal{E}_t(Z_{t+1}),\quad t=0,\dotso,T,
\end{equation}
and $\tau^{**}:=\inf\{t\geq 0:\ f_t=Z_t\}$ is optimal.
\end{remark}

\subsection{An example in which $\mathcal{E}_t[Y_{t+1}]$ is Borel measurable}
Let $S=(S_i)_{i=1}^T$ be the canonical process and $\mathcal{R}$ be the set of martingale measures on some compact set $\mathcal{K}\subset\Omega_T$. Assume $\mathcal{R}\neq\emptyset$. Then for $\omega\in\mathcal{K}$, $\mathfrak{R}_t(\omega^t)$ is the set of martingale measures on $\mathcal{K}$ from time $t$ to $T$ given the previous path $\omega^t$. \prref{exa} below indicates that the assumption in \thref{infsup} is satisfied provided $f$ is u.s.c. in $\omega$. 

\begin{proposition}\label{exa}
Assume that $f_t$ is u.s.c. for $t=1,\dotso,T$. Then $\mathcal{E}_t[Y_{t+1}]$ is u.s.c., and thus $\mathcal{B}(\Omega_t)$-measurable, $t=1,\dotso,T$.
\begin{proof}
Since $\mathcal{K}$ is compact, it is easy to check that the set $\{(\omega,P):\ \omega\in\mathcal{K}, P\in\mathfrak{R}_t(\omega^t)\}$ is closed. By \cite[Proposition 7.33]{Shreve}, $Y_t$ defined in \eqref{bb1} is u.s.c. Following the proof similar to that of \prref{semicti}, it could be shown that for $(\omega,P)\in\Omega_T\times\mathfrak{P}(\Omega_{T-t})$, the map $(\omega,P)\mapsto E_P[Y(\omega^t,\cdot)]$ is u.s.c. Then applying \cite[Proposition 7.33]{Shreve} again, we know that $\mathcal{E}_t[Y_{t+1}]$ is u.s.c.
\end{proof}
\end{proposition}

\section{Upper-semianalyticity and the super-martingale property}\label{sec:appC}
The result in this section is used in the proof of \thref{supheg}. Let us use the setting in Section B. Let $\phi=(\phi_t)_{t=0}^T$ be an adapted process, and $\mathfrak{g}$ be u.s.a. Define the process $U=(U_t)_{t=0}^T$ as
\begin{equation}\label{ap10}
U_t:=\sup_{\tau\in\mathcal{T}_t}\mathcal{E}_t[\phi_\tau+\mathfrak{g}].
\end{equation}
We have the following result.
\begin{proposition}\label{cc3}
Assume for $(\omega,P)\in\Omega_T\times\mathfrak{P}(\Omega_{T-t})$, the map $(\omega,P)\mapsto\sup_{\tau\in\mathcal{T}_t}E_P[\phi_\tau(\omega^t,\cdot)]$ is u.s.a., $t=1,\dotso,T$. Then $U_t$ defined in \eqref{ap10} is u.s.a. and $\mathcal{F}_t$-measurable for $t=1,\dotso,T$, and $U=(U_t)_{t=0}^T$ is a super-martingale under each $P\in\mathcal{R}$.
\end{proposition}
\begin{proof}
Using the fact that the map $(\omega,P)\mapsto E_p[\mathfrak{g}(\omega^t,\cdot)]$ is u.s.a. for $(\omega,P)\in\Omega_T\times\mathfrak{P}(\Omega_{T-t})$ (see the last paragraph on page 8 in \cite{Nutz3}), we deduce that the map $(\omega,P)\mapsto\sup_{\tau\in\mathcal{T}_t}E_P[\phi_\tau(\omega^t,\cdot)+\mathfrak{g}(\omega^t,\cdot)]$ is u.s.a. Since $\mathfrak{R}_t(\omega^t)$ is the $\omega$-section of an analytic set, we can apply \cite[Proposition 7.47]{Shreve} to conclude that $U_t$ is u.s.a., $t=1,\dotso,T$.  As $U_t$ only depends on the path up to time $t$, it is $\mathcal{F}_t$-measurable.

In the rest of the proof, we shall show that
\begin{equation}\label{cc1}
U_t\geq\mathcal{E}_t[U_{t+1}],
\end{equation}
which will imply the super-martingale property of $U$ under each $P\in\mathcal{R}$.
Fix $(t,\omega)\in\{0,\dotso,T\}\times\Omega_T$ and let $P=P_t\otimes\dotso\otimes P_{T-1}\in\mathfrak{R}_t(\omega^t)$. For any $\eps>0$, since the map $(\tilde\omega,P)\mapsto\sup_{\tau\in\mathcal{T}_t}E_P[\phi_\tau(\omega^t,\tilde\omega,\cdot)+\mathfrak{g}(\omega^t,\tilde\omega,\cdot)]$ is u.s.a. for $(\tilde\omega,P)\in\Omega_1\times\mathfrak{P}(\Omega_{T-t-1})$, and $\mathfrak{R}_{t+1}(\omega^t,\tilde\omega)$ is the $\tilde\omega$-section of an analytic set, we can apply theorem \cite[Proposition 7.50]{Shreve} and get that there exists a universally measurable selector $P^\eps(\omega^t,\cdot)$, such that $P^\eps(\omega^t,\tilde\omega)=P_{t+1}^\eps(\omega^t,\tilde\omega)\otimes\dotso\otimes P_{T-1}^\eps(\omega^t,\tilde\omega,\cdot)\in\mathfrak{R}_{t+1}(\omega^t,\tilde\omega)$, and
$$\left(\sup_{\tilde P\in\mathfrak{R}_{t+1}(\omega^t,\tilde\omega)}\sup_{\tau\in\mathcal{T}_{t+1}} E_{\tilde P}[\phi_\tau(\omega^t,\tilde\omega,\cdot)+\mathfrak{g}(\omega^t,\tilde\omega,\cdot)]-\eps\right)1_A+\frac{1}{\eps}1_{A^c}\leq\sup_{\tau\in\mathcal{T}_{t+1}}E_{P^\eps(\omega^t,\tilde\omega)}[\phi_\tau(\omega^t,\tilde\omega,\cdot)+\mathfrak{g}(\omega^t,\tilde\omega,\cdot)],$$
where 
$$A=\{\tilde\omega\in\Omega_1:\ \sup_{\tilde P\in\mathfrak{R}_{t+1}(\omega^t,\tilde\omega)}\sup_{\tau\in\mathcal{T}_{t+1}} E_{\tilde P}[\phi_\tau(\omega^t,\tilde\omega,\cdot)+\mathfrak{g}(\omega^t,\tilde\omega,\cdot)]<\infty\}.$$
Define
$$P^*:=P_t\otimes P_{t+1}^\eps\otimes\dotso\otimes P_{T-1}^\eps\in\mathfrak{R}_t(\omega^t).$$
Then we have that
\begin{eqnarray}
&&\hspace{-2cm}E_P\left[\left(U_{t+1}(\omega^t,\cdot)-\eps\right)1_A+\frac{1}{\eps}1_{A^c}\right]\notag\\
&&=E_P\left[\left( \sup_{\tilde P\in\mathfrak{R}_{t+1}(\omega^t,\tilde\omega)}\sup_{\tau\in\mathcal{T}_{t+1}} E_{\tilde P}[\phi_\tau(\omega^t,\tilde\omega,\cdot)+\mathfrak{g}(\omega^t,\tilde\omega,\cdot)]-\eps\right)1_A+\frac{1}{\eps}1_{A^c}\right]\notag\\
&&\leq E_P\left[\sup_{\tau\in\mathcal{T}_{t+1}}E_{P^\eps(\omega^t,\tilde\omega)}[\phi_\tau(\omega^t,\tilde\omega,\cdot)+\mathfrak{g}(\omega^t,\tilde\omega,\cdot)]\right]\notag\\
&&=E_{P^*}\left[\sup_{\tau\in\mathcal{T}_{t+1}}E_{P^\eps(\omega^t,\tilde\omega)}[\phi_\tau(\omega^t,\tilde\omega,\cdot)+\mathfrak{g}(\omega^t,\tilde\omega,\cdot)]\right]\notag\\
&&=E_{P^*}\left[\sup_{\tau\in\mathcal{T}_{t+1}}E_{P^\eps(\omega^t,\tilde\omega)}[\phi_\tau(\omega^t,\tilde\omega,\cdot)]\right]+E_{P^*}[\mathfrak{g}(\omega^t,\cdot)]\notag\\
&&\leq\sup_{\tau\in\mathcal{T}_t}E_{P^*}[\phi_\tau(\omega^t,\cdot)]+E_{P^*}[\mathfrak{g}(\omega^t,\cdot)]\notag\\
&&\leq U_t(\omega),\notag
\end{eqnarray}
where the fourth line follows from the fact that $P^*=P$ from time $t$ to $t+1$, the fifth line follows from the tower property as $P^*=P_t\otimes P^\eps$, and the sixth line follows from the classical optimal stopping theory under a single probability measure $P^*$. As $t,\omega,P$ and $\eps$ are arbitrary, \eqref{cc1} holds.
\end{proof}

\section{No arbitrage when there are no options for static hedging}\label{section:appD}
Let $S=(S_t)_{t=0,\dotso,T}$ be the canonical process taking values in some path space $\mathcal{K}\subset\{1\}\times\mathbb{R}^T$, which represents the stock price process. We take $\mathcal{P}$ to be the set of all the probability measures on $\mathcal{K}$. In this secton, we assume that there is no hedging option available, i.e., $e=0$. Let us first identify the reasonable path spaces:
 
\begin{definition} $\mathcal{K}\subset\{1\}\times\mathbb{R}^T$ is called a reasonable path space, if for any $t\in\{0,\dotso,T\}$ and $(s_0=1,s_1,\dotso,s_T)\in\mathcal{K}$, 
\begin{itemize}
\item[(i)] if $s_t>0$, then there exists $(s_0,\dotso,s_t,,s_{t+1}^i,\dotso,s_T^i)\in\mathcal{K},\ i=1,2$, such that $s_{t+1}^1< s_t<s_{t+1}^2$;
\item[(ii)] if $s_t=0$, then $s_k=0,\ k\geq t+1$.
\end{itemize}
\end{definition}

Obviously, if $\mathcal{K}$ is a reasonable path space, then a martingale measure on $\mathcal{K}$ is easy to construct, and thus the no arbitrage in \cite{Schachermayer2} is satisfied. The following proposition states that NA$(\mathcal{P})$ also holds. So the no arbitrage definitions in \cite{Schachermayer2} and \cite{Nutz2} in fact coincide in the case when $\mathcal{K}$ is a reasonable path space and $e=0$. 
\begin{proposition}\label{prop:no-opt}
If $\mathcal{K}$ is a reasonable path space, then NA$(\mathcal{P})$ holds.
\end{proposition}
\begin{proof}
Let $H=(H_0,\dotso,H_{T-1}(s_1,\dotso,s_{T-1}))$ be a trading strategy such that 
\begin{equation}\label{b1}
(H\cdot S)_T\geq 0, \quad\forall s\in\mathcal{K}.
\end{equation}
We need to show $(H\cdot S)_T=0, \forall s\in\mathcal{K}$. It suffices to show that 
\begin{equation}\label{b2}
H_k(s_1,\dotso,s_k)=0, \quad\text{for } s_k>0,
\end{equation}
for $k=0,\dotso,T-1$. We shall show \eqref{b2} by the induction. 

Assume $H_0\neq 0$. Then take $s_1^*>s_0$ if $H_0<0$, and take $s_1^*<s_0$ if $H_0>0$. In general, for $j=1,\dotso T-1$, take $s_{j+1}^*\geq s_j^*$ if $H(s_1^*,\dotso,s_j^*)\leq0$ and $s_{j+1}^*\leq s_j^*$ if $H(s_1^*,\dotso,s_j^*)>0$. Then $(H\cdot S)_T(s_0,s_1^*,\dotso,s_T^*)<0$, which contradicts \eqref{b1}. Hence $H_0=0$ and \eqref{b2} holds for $k=0$.

Assume \eqref{b2} holds for $k\leq t-1$. Then for any $(s_0,\dotso,s_t)$ with $s_t>0$, by assumption (ii), we have that $s_i>0,\ i=0,\dotso,t-1$, and thus $H_i(s_1,\dotso,s_i)=0,\ i=0,\dotso,t-1$ by the induction hypothesis. If $H_t(s_1,\dotso,s_t)\neq 0$, then we can similarly construct $(s_{t+1}^*,\dotso,s_T^*)$ as above, such that $(H\cdot S)_T(s_0,\dotso,s_t,s_{t+1}^*,\dotso,s_T^*)<0$, which contradicts \eqref{b1}. Hence $H_t(s_1,\dotso,s_t)=0$ and \eqref{b2} holds for $k=t$.
\end{proof}

\end{appendices}

\bibliographystyle{siam}
\bibliography{ref}

\end{document}